\documentclass[a4paper,12pt]{amsart}

\usepackage[latin1]{inputenc}
\usepackage[T1]{fontenc}
\usepackage[english]{babel}

\usepackage{mathrsfs}
\usepackage{amscd}
\usepackage{amsfonts}
\usepackage{amsmath}
\usepackage{amssymb}
\usepackage{amstext}
\usepackage{amsthm}
\usepackage{amsbsy}

\usepackage{xspace}
\usepackage[all]{xy}
\usepackage{graphicx}
\usepackage{url}
\usepackage{latexsym}

\makeatletter
\newcommand*{\rom}[1]{\expandafter\@slowromancap\romannumeral #1@}
\makeatother

\theoremstyle{definition}

\newtheorem{fact}{fact}

\newtheorem{thm}[fact]{Theorem}
\newtheorem{lemma}[fact]{Lemma}
\newtheorem{prop}[fact]{Proposition}
\newtheorem{corollary}[fact]{Corollary}
\newtheorem{defini}[fact]{Definition}

\title{The distribution of $ITRM$-recognizable reals}
\author{Merlin Carl}

\begin{document}

\begin{abstract}
Infinite Time Register Machines ($ITRM$'s) are  a well-established machine model for infinitary computations. Their computational strength relative to oracles is understood, see e.g. \cite{Koe}, \cite{KoeWe} and \cite{KoeMi}. 
We consider the notion of recognizability, which was first formulated for Infinite Time Turing Machines in \cite{HamLew} and applied to $ITRM$'s in \cite{ITRM}. A real $x$ is $ITRM$-recognizable
iff there is an $ITRM$-program $P$ such that $P^{y}$ stops with output $1$ iff $y=x$, and otherwise stops with output $0$. In \cite{ITRM}, it is shown that the recognizable reals are not contained in the computable reals.
Here, we investigate in detail how the $ITRM$-recognizable reals are distributed along the canonical well-ordering $<_{L}$ of G\"odel's constructible hierarchy $L$. In particular, we prove that the recognizable reals have gaps in $<_{L}$,
that there is no universal $ITRM$ in terms of recognizability and consider a relativized notion of recognizability. 
\end{abstract}

\maketitle

\section{Preliminaries}
Infinite Time Register Machines ($ITRM$'s) are a machine model for infinite computations introduced by Peter Koepke and Russell Miller in \cite{KoeMi}. We will describe this model only shortly. Detailed descriptions of $ITRM$'s and
all of the results about these machines we will use here can be found in \cite{KoeMi} and \cite{ITRM}.\\
An $ITRM$ resembles in most of its features a classical universal register machine ($URM$) from \cite{Cutl}: It has finitely many registers $R_{1},...,R_{n}$ each of which can store one natural number. An $ITRM$-program consists
of finitely many lines, each of which contains one command. Commands are the increasing and decreasing of a register content by $1$, copying a register content to another register, evaluating an oracle, jumping to a certain program line provided a certain register 
content is $0$, and stopping.\\
In contrast to $URM$'s, $ITRM$'s allow an arbirary ordinal as their running time. Accordingly, the definition of an $ITRM$-computation now has to take care of limit steps. At successor ordinals, we define the computation in the same way as for $URM$'s.
If $\lambda$ is a limit ordinal, we set the content $R_{i\lambda}$ of the $i$-th register $R_{i}$ at time $\lambda$ to $liminf_{\iota<\lambda}R_{i\iota}$ iff this limit exists, and to $0$ otherwise. Likewise, the active program line $Z_{\lambda}$
to be carried out in the $\lambda$th step is $liminf_{\iota<\lambda}Z_{\iota}$, where the limit always exists as the set of lines is finite and their indices are therefore bounded.\\

\begin{defini}
 $x\subseteq\omega$ is $ITRM$-computable in the oracle $y\subseteq\omega$ iff there exists an $ITRM$-program $P$ such that, for $i\in\omega$, $P$ with oracle $y$ stops whatever natural number $j$ is in its first register at the start of the computation and returns
$1$ iff $j\in x$ and otherwise returns $0$. A real computable in the empty oracle is simply called computable. 
\end{defini}

Apart from computability, which is a direct analogue of the corresponding finite concept, there is a different notion of how an $ITRM$ can 'handle' a real number, which has no interesting analogue in the finite. A classical $URM$ $R$ can only process a finite
part of each oracle, and hence, for each real $r$, there is an open environment $u$ of $r$ such that $R$ cannot distinguish the elements of $u$. The computing time of an $ITRM$, on the other hand, allows it to repeatedly consider each bit of a real number. 
Hence, it has a chance of identifying individual real numbers. Numbers for which this is possible are called 'recognizable'.

\begin{defini}
Let $r\in\mathfrak{P}(\omega)$. Then $r$ is recognizable iff there is an $ITRM$-program $P$ such that $P^{x}$ stops with output $1$ iff $x=r$ and otherwise stops with output $0$.
\end{defini}

This work originated in \cite{ITRM} with the proof of the Lost Melody Theorem for $ITRM$'s. The basic results on gaps and the idea of relativization are contained in \cite{Carl}. 
Most of the other results were obtained at or shortly after CiE $2012$, where this work was presented (see \cite{Recog}). I thank Philip Welch for two short, but very helpful conversations
on this topic.

Most of our notation is standard. $ZF^{-}$ is $ZF$ set theory without the power set axiom. $\mathfrak{P}(x)$ will denote the power set of $x$. For an $ITRM$-program $P$, $P^{x}(i)\downarrow=j$ means that the program $P$ with oracle $x$ with initial input $i$
in its first register stops with output $j$ in register $1$. We take $R_1$ to be the generic register for input and output and will not care about such details in the further course of this paper. 
By $\omega_{\iota}^{CK}$, we denote the $\iota$-th admissible ordinal, where $\iota\in On$. When we consider admissible ordinals relative to a real $x$, we write $\omega_{\iota}^{CK,x}$. For $X\subseteq L_{\alpha}$, $\Sigma_{1}^{L_{\alpha}}\{X\}$ 
denotes the $\Sigma_{1}$-Skolem hull of $X$ in $L_{\alpha}$. When $H$ is a $\Sigma_{1}$-substructure of some $L_{\alpha}$, then $\pi:H\equiv L_{\gamma}$ and $\pi:H\rightarrow_{coll}L_{\gamma}$ denote the Mostowski collapse of $H$ to $L_{\gamma}$ with isomorphism $\pi$. Throughout the paper, $p:\omega\times\omega\rightarrow\omega$ denotes the canonical bijection
between $\omega\times\omega$ and $\omega$.

\section{Fine structure and Lost Melodies}

We recall some basic facts about the fine structure of G\"odel's constructible universe $L$, $ITRM$-computability and recognizable reals. The results on $ITRM$s can be found in \cite{Koe}, \cite{KoeMi} and \cite{ITRM}. The canonical source for $L$ is \cite{Jen}.\\

\begin{thm}{\label{compstrength1}}
There is a function $g:\omega\rightarrow\omega$ such that $x$ is computable by an $ITRM$-program using $g(n)$ registers in the oracle $y$ if $x\in L_{\omega_{n}^{CK,y}}[y]$.
\end{thm}
\begin{proof}
This is a relativized version of the main result of \cite{KoeMi}. It is not hard, but rather tedious, to check that the proof given there relativizes as well.
\end{proof}

\begin{thm}{\label{compstrength2}}
If $x$ is computable in the oracle $y$ by an $ITRM$-program using $n$ registers, then $x\in L_{\omega_{1}^{CK,y}}$.
\end{thm}

\begin{thm}{\label{haltingtimes}}
A program $P$ using $n$ registers using oracle $x$ either halts after at most $\omega_{n+1}^{CK,x}$ many steps or does not halt at all. 
\end{thm}

\begin{thm}{\label{haltingproblem}}
Let $n\in\omega$ and let $(P_{i}|i\in\omega)$ be a canonical enumeration of the $ITRM$-programs using at most $n$ registers. Then there is an $ITRM$-program $H_{n}$ such that, for all $x\subseteq\omega$,
$H_{n}^{x}(i)\downarrow=1$ iff $P_{i}^{x}\downarrow$, and $H_{n}^{x}(i)\downarrow=1$, otherwise.
\end{thm}

\begin{defini}
We denote by $COMP$ the set of computable reals and by $RECOG$ the set of recognizable reals.
\end{defini}

\begin{lemma}
$COMP\subseteq RECOG$
\end{lemma}
\begin{proof}
Let $x$ be computable, and suppose that $P$ is a program that computes $x$. Then we can recognize $x$ as follows: Let $R_{1}$ and $R_{2}$ be two extra registers, which we call the 'flag registers'. Let $R_{1,0}=1$ and $R_{2,0}=0$. 
For each $i\in\omega$, compute $P$ on input $i$, then compare the output to the $i$-th bit of the oracle. If the results differ, return 'no'.
If the results match, set $R_{1}$ to $0$ and then back to $1$ again and set $R_{2}$ to $1$ and then back to $0$. If the contents of $R_{1}$ and $R_{2}$ are equal at the beginning of such a step,
stop with output 'yes'. This works because the contents of the flag registers can only agree when they have both been flashed a limit number of times, which means that all bits of oracle have been positively checked.
\end{proof}

\begin{defini}
For $\alpha\in On$, $i\in\omega+1$, $\rho_{i}^{\alpha}$ denotes the $\Sigma_{i}$-projectum of $L_{\alpha}$, i.e. the smallest ordinal $\rho$ such that there is some $x\subset\rho$ with $x\in\Sigma_{i}(L_{\alpha})-L_{\alpha}$. $\rho_{\omega}^{\alpha}$ denotes
the ultimate projectum of $L_{\alpha}$, i.e. the minimal element of $\{\rho_{i}^{\alpha}|i\in\omega\}$.
\end{defini}

\begin{lemma}{\label{finestructure}}
Suppose that $\alpha$ is an ordinal such that, for some set $S$ of $\in$-formulas, $\alpha$ is minimal with the property that $L_{\alpha}\models S$. Then $\rho_{\omega}^{\alpha+i}=\omega$ for some $i\in\{0,1\}$.
\end{lemma}
\begin{proof}
For $X\subseteq L_{\alpha}$, let $\Sigma_{\omega}^{L_{\alpha}}\{X\}$ be the elementary hull of $X$ in $L_{\alpha}$. Hence $\Sigma_{\omega}^{L_{\alpha}}\{X\}$ is the hull of $X$ under all $\Sigma_{i}$-Skolem functions for $L_{\alpha}$. Denote by $h_{i}^{L_{\alpha}}$
the $\Sigma_{i}$-Skolem function for $L_{\alpha}$. Consider $H:=\Sigma_{\omega}^{L_{\alpha}}\{\emptyset\}$. By condensation, there is $\beta\leq\alpha$ such that $\sigma:H\equiv L_{\beta}$. By elementarity, $L_{\beta}\models S$. By minimality, $\beta=\alpha$
and hence indeed $\sigma=id$. The $\Sigma_{\omega}$-Skolem function $h_{\omega}^{L_{\alpha}}$ is definable over $L_{\alpha+1}$. Now $h_{\omega}^{L_{\alpha}}[\omega]=L_{\alpha}$. Hence, over $L_{\alpha+1}$, we can define a surjection $f$ from $\omega$ onto 
$L_{\alpha}$. Now consider $Y:=\{i\in\omega|i\notin f(i)\}$. If $Y\in L_{\alpha}$, then there is $j\in\omega$ with $f(j)=Y$ and hence we have $j\in f(j)\leftrightarrow j\notin Y\leftrightarrow j\notin f(j)$, a contradiction. Hence $Y\notin L_{\alpha}$.
But $Y$ is obviously definable over $L_{\alpha+1}$, hence $Y\in L_{\alpha+2}-L_{\alpha}$. So at least one of $\rho_{\omega}^{L_{\alpha}}$ and $\rho_{\omega}^{L_{\alpha+1}}$ must drop to $\omega$.
\end{proof}

\begin{lemma} Let $P$ be an $ITRM$-program. Denote by $comp(P)$ the computation of $P$, i.e. the sequence of program states when carrying out $P$ as defined in \cite{KoeMi}.
\begin{itemize}
\item[i] Let $\alpha\in On$. Denote by $comp(P)\upharpoonright\alpha$ the computation associated with $P$ restricted to the first $\alpha$ many steps. Let $\alpha<\beta,\gamma$ be ordinals. 
Then $(comp(P)\upharpoonright\alpha)^{L_{\beta}}=(comp(P)\upharpoonright\alpha)^{L_{\gamma}}$.
\item[ii] Let $\alpha>\beta\geq\omega_{\omega}^{CK}$. Then $(comp(P))^{L_{\alpha}}=(comp(P))^{L_{\beta}}$.
\item[iii] Let $M_{1}$ and $M_{2}$ be transitive models of $ZF^{-}$. Then $(comp(P))^{M_{1}}=(comp(P))^{M_{2}}$.
\end{itemize}
\end{lemma}
\begin{proof}
(i) An easy transfinite induction on $\alpha$.\\ 
(ii) This follows from (i) and the fact that $ITRM$-programs either halt in less than $\omega_{\omega}^{CK}$ many steps or do not halt at all.\\
(iii) This follows from (ii) and the fact that $ZF^{-}$ proves the existence of $\omega_{\omega}^{CK}$.
\end{proof}

The following uses a canonical way to encode countable $\in$-structures as reals. We do this by the following definition. Fix some canonical recursive bijection $p$ between $\omega$ and $\omega\times\omega$ and denote, for $a\in\omega$, by $(a)_{1}$ and $(a)_{2}$
the first and second component of $p(a)$, respectively.

\begin{defini}
Let $X\in L$ be countable, and let $E\in L$ be a binary relation on $X$. The canonical code of $(X,E)$, $cc(X,E)$, is the $<_{L}$-smallest real $r$ such that, for some bijective $f:\omega\rightarrow X$, we have
$a\in r\leftrightarrow E(f((a)_{1}),f((a)_{2})$.
\end{defini}

The following is called the Lost Melody Theorem for $ITRM$'s:\\

\begin{thm}{\label{LoMe}}
$COMP\subsetneq RECOG$, i.e. there are reals which are recognizable, but not computable.
\end{thm}
\begin{proof}
$COMP\subseteq RECOG$ was shown above. It remains to see that $COMP\neq RECOG$.\\
We only sketch the construction, as we will re-use it in various modifications below. The detailed proof is quite technical and can be found in \cite{ITRM}.\\
Let $\alpha$ be minimal such that $L_{\alpha}\models ZF^{-}$, and let $x=cc(L_{\alpha})$. It is easily seen that
$\Sigma_{\omega}^{L_{\alpha+1}}\{\{L_{\alpha}\}\}=L_{\alpha+1}$, hence there is $f\in L_{\alpha+2}$ such that $f:\omega\leftrightarrow L_{\alpha}$, and hence also $x\in L_{\alpha+2}$.\\
We claim that $x$ is recognizable, but not computable.\\
That $x$ is not computable is easy to see: If it was, the computation could, by the last lemma, be carried out inside $L_{\alpha}$ with the same effect. This would allow us to define $x$ 
over $L_{\omega_{\omega}^{CK}}\in L_{\alpha}$, hence we would have $x\in L_{\alpha}$. But then we could decode $x$ inside $L_{\alpha}$, which leads $L_{\alpha}\in L_{\alpha}$, a contradiction.\\
Now we argue that $x$ is recognizable:\\
By a central result from \cite{KoeMi}, we can test whether the relation coded by a given real number $x$ is well-founded.\\
After this, it is not hard to check whether certain first-order $\in$-formulas hold in the structure coded by $x$. This is obvious for atomic formulas. It is also easy to see
that we can test for $\neg\phi$ and $\phi\wedge\psi$ if we can test for $\phi$ and $\psi$.\\
To test whether $\exists{x}\phi$ holds, we search through $\omega$ for an example, applying the checking procedure for $\phi$ in each case. Some trickery is necessary
to check the truth of formulas of arbitrary quantifier complexity by a single program, this can be found in \cite{ITRM}. Since the axioms of $ZF^{-}+V=L$ form a recursive set,
we can test whether they all hold inside the structure coded by $x$. We can also check that no element of this structure is a well-founded model of $ZF^{-}+V=L$, and that it must
hence be $\in$-minimal.\\
The hard part of the proof is checking the $<_{L}$-minimality of $x$: After all, $x$ cannot be an element of $L_{\alpha}$, hence evaluating the truth predicate for $L_{\alpha}$ will not
help much. The idea is that each element of $L_{\alpha+2}$ is constructed from elements of $X:=L_{\alpha}\cup\{L_{\alpha}\}\cup\{L_{\alpha+1}\}$ by repeated application of G\"odel functions
(see e.g. \cite{Jech}, \cite{Jen}). Hence, we can name these elements by a finite sequence of G\"odel functions and a finite sequence of codes for elements of $X$. Formulas can then be 
evaluated by recursion on the complexity of the names occuring in it. The details of this recursion can be found in \cite{ITRM}.
\end{proof}

\begin{corollary}{\label{focheck}}
Let $(\phi_{i}|i\in\omega)$ be a canonical enumeration of the $\in$-formulas. There is an $ITRM$-program $P$ such that, for all $x\subseteq\omega$, $i\in\omega$, $\vec{v}=(v_{1},...,v_{n})$ a finite sequence of natural numbers of the appropriate length
coded by a natural number $\bar{v}$, $P^{x}(i,\vec{v})\downarrow=1$ iff $x$ codes some $L_{\alpha}$ such that $\phi_{i}(x_{1},...,x_{n})$ holds in $L_{\alpha}$, where $x_{1},...,x_{n}$ are the elements coded by $v_{1},...,v_{n}$, respectively, 
and otherwise $P^{x}(i,\vec{v})\downarrow=0$. The same holds with a recursive set $S$ of formulas instead of one single formula $\phi$, where $i$ is then a code for a Turing program enumerating $S$.
\end{corollary}
\begin{proof}
This follows immediately from the proof of Theorem \ref{LoMe}.
\end{proof}

\textbf{Remarks}: (1) For technical reasons, the proof in \cite{ITRM} uses Jensen's $J$-hierarchy instead of the $L$-hierarchy. As this distinction is irrelevant for the results in our paper,
we decided to switch to the more familiar G\"odel hierarchy instead.\\
(2) From this construction, we also get a general procedure for evaluating truth critera for elements of $L_{\alpha+i}$ for every $i\in\omega$. We could as well go further (though it is not
immediately clear how far), but this will be of no relevance for this article.\\
(3) In \cite{WITRM}, a weaker type of $ITRM$ with a modified limit rule is discussed: If $liminf_{\iota<\lambda}R_{i\iota}$ does not exist for some $\lambda$ and some register $R_{i}$, then the next computation step is undefined and the computation fails.
The computable reals for these weak $ITRM$'s are exactly those in $L_{\omega_{1}^{CK}}$. We do not know whether there are Lost Melodies for weak $ITRM$'s.

\section{Basic Results on the Distribution of Recognizable Reals}

\begin{thm}{\label{shoenfield}} 
$RECOG\subseteq L$
\end{thm}
\begin{proof}
Let $x\in RECOG$, and suppose that $P$ is an $ITRM$-program recognizing $x$. The statement $\phi(P):=\exists{y\in\mathfrak{P}(\omega)}:P^{y}\downarrow=1$ is $\Sigma_{2}^{1}$, and hence, by Shoenfield's absoluteness theorem, 
absolute for transitive models of $ZFC$. Hence $L\models\phi(P)$.
Therefore, there must be $z\in\mathfrak{P}(\omega)\cap{L}$ such that $P^{z}$ stops with output $1$ in $L$. Since computations are absolute between transitive models of $ZFC$, $P^{z}\downarrow=1$. Since $P$ recognizes $x$, this only happens for $x=z$.
Hence $x\in L$.
\end{proof}

In fact, a nonconstructible real is quite far away from being recognizable:

\begin{thm}{\label{mansfield}}
If $x\in\mathfrak{P}(\omega)-L$ and $P$ is an $ITRM$ program such that $P^{x}\downarrow=1$, then $P^{y}\downarrow=1$ holds for a perfect set of real numbers.
\end{thm}
\begin{proof}
By the Mansfield-Solovay Theorem (see e.g. \cite{Mi}), a $\tilde{\Sigma}_{2}^{1}$-set with constructible parameters $A\subseteq\omega^{\omega}$ which contains some $r\in\mathfrak{P}(\omega)-L$ contains a perfect set as a subset. Given an $ITRM$-program $P$,
the statement that $P^{x}$ stops with output $1$ is obviously $\tilde{\Sigma}_{2}^{1}$. Hence, if an $ITRM$-program $P$ stops with output $1$ on a certain non-constructible oracle $r$, it will do so on all reals in a perfect set. This contradicts the definition
of recognizability.
\end{proof}

Hence the computable reals are a proper subset of the recognizable reals, which in turn are a proper subset of the constructible reals. We now turn our attention to the question where recognziable reals appear
in the $L$-hierarchy.

\begin{lemma}
 Let $L_{\alpha}\models ZF^{-}$, $x\in L_{\alpha}$ such that $L_{\alpha}\models \neg{RECOG(x)}$. Then $\neg{RECOG(x)}$.
\end{lemma}
\begin{proof}
Let $P$ be an $ITRM$-program. If $P^{x}$ stops, then it does so in less than $\omega_{\omega}^{x,CK}$ many steps (see \cite{KoeMi}). Hence $P^{x}$ stops inside $L_{\alpha}$ or does not stop at all. As computations are absolute, the result of the computation is absolute between
$L_{\alpha}$ and $V$. If $P$ does not recognize $x$ inside $L_{\alpha}$, then either $P^{x}\downarrow=0$ within $L_{\alpha}$, and hence inside $V$; or $P^{r}\uparrow$ for some $r\in L_{\alpha}\cap\mathfrak{P}(\omega)$ inside $L_{\alpha}$, and hence in $V$;
or $P^{x}\downarrow=P^{r}\downarrow=1$ for some $x\neq r\in L_{\alpha}\cap\mathfrak{P}(\omega)$ in $L_{\alpha}$ and hence in $V$. In each case, $x$ is not recognizable.
\end{proof}


\begin{thm}{\label{gaps}}
There is a gap in the recognizable reals, i.e. there are constructible reals $x<_{L}y<_{L}z$ such that $x$ and $z$ are recognizable, but $y$ is not.
\end{thm}
\begin{proof}
As $\{\alpha<\omega_{1}|L_{\alpha}\models ZF^{-}\}$ is cofinal in $\omega_{1}$, but $|RECOG|=\omega$, there must be a minimal $\gamma$ such that $L_{\gamma}\models ZF^{-}+\exists{x\subseteq\mathfrak{P}(\omega)}\neg RECOG(x)$. Let $z=cc(L_{\gamma})$.
Using the same argument as for the Lost Melody theorem for $ITRM$'s, we see that $z\in RECOG$. Now let $\omega\supset y\in L_{\gamma}$ such that $L_{\gamma}\models\neg RECOG(y)$. By our lemma above, it follows that $\neg RECOG(y)$.
Hence $0<_{L}y<_{L}z$ witnesses our claim.
\end{proof}

The next natural question is how large these gaps can become.

\begin{defini}
Let $\delta$ be an ordinal. A $\delta$-gap is a $<_{L}$-intervall $[x,y[$ of constructible reals of order type $\delta$ such that no element of $[x,y[$ is constructible, while $y$ is.\\
A strong $\delta$-gap is an intervall of ordinals $[\alpha,\alpha+\delta[$ such that, for no $\beta\in[\alpha,\alpha+\delta[$, $L_{\beta}$ contains a recognizable ordinal, while $L_{\alpha+\delta}$ contains one.
\end{defini}

The following results hold for both gaps and strong gaps, each time by an obvious modification of the proof. Note that the relation between the two notions is not trivial: In particular, a strong gap may not be a gap at all,
since the fact that it does not contain recognizable reals may be due to the fact that it contains no reals at all. Since this is not the kind of phenomenon we are interested in, we make the following definition:

\begin{defini}
 Let $G:=[\alpha,\beta[$ be a strong gap, where $\beta$ is a limit ordinal. $G$ is called substantial, if, for any $\gamma\in G$, there is $\gamma<\gamma^{\prime}\in G$ such that $\mathfrak{P}^{L}(\omega)\cap(L_{\gamma+1}-L_{\gamma})\neq\emptyset$.
$G$ is weakly substantial iff, for every $\gamma\in G$, there is $\delta\in G$ such that $\delta>\gamma$ and $(L_{\delta+1}-L_{\delta})\cap\mathfrak{P}(\omega)\neq\emptyset$.
\end{defini}

\begin{defini}
Let $x,y\in\mathfrak{P}(\omega)$. Then $x\oplus y\subseteq\omega$ is defined by $\forall_{i\in\omega}((2i\in x\oplus y\leftrightarrow i\in x)\wedge(2i+1\in x\oplus y\leftrightarrow i\in y))$.
\end{defini}

\begin{prop}
If $a$ and $b$ are recognizable, then so is $a\oplus b$. However, there are $a$ and $b$ such that $a\notin RECOG$, but $a\oplus b\in RECOG$.
\end{prop}
\begin{proof}
The first statement is obvious: To recognize $a\oplus b$, simply apply the recognizing procedures for $a$ and $b$ seperately to the even and even bits of a given real, respectively. For the second statement, let $a$ and $b$ be $y$ and $z$ from the proof 
of the last theorem, respectively.
\end{proof}

\textbf{Question}: Are there nonrecognizable constructible reals $a$ and $b$ such that $a\oplus b\in RECOG$?

\begin{thm}{\label{smallgaps}}
For each $\delta<\omega_{1}^{CK}$, there is a (strong) $\delta$-gap.
\end{thm}
\begin{proof}
For every such $\delta$, there is a minimal $\gamma(\delta)$ such that $L_{\gamma(\delta)}\models ZF^{-}+V=L$ and $L_{\gamma(\delta)}$ contains a $<_{L}$-intervall of reals with order type $\delta$ and all elements unrecognizable. 
This is again because there are cofinally many $\alpha$ with $L_{\alpha}\models ZF^{-}+V=L$ in $\omega_{1}$, but only boundedly many recognizable reals. As $\delta<\omega_{1}^{CK}$, there is a recursive (in the classical sense) 
real $x\subseteq\omega$ coding $\delta$. Furthermore, let $r^{\prime}=cc(L_{\gamma(\delta)})$.\\
We claim that $x\oplus r^{\prime}$ is recognizable. This suffices, for it then follows that $x\oplus r^{\prime}$ is $>_{L}$ all elements of the $\delta$-intervall of unrecognizables in $L_{\gamma(\delta)}$. (If $x\oplus r^{\prime}$ was an element of 
$L_{\gamma(\delta)}$, then we could decode $r^{\prime}$ in $L_{\gamma(\delta)}$ and obtain $L_{\gamma(\delta)}\in L_{\gamma(\delta)}$, a contradiction.)\\
Now for the claim: By the last proposition, it suffices to show that both $x$ and $r^{\prime}$ are recognizable. As $x$ is recursive, it is computable, and hence recognizable.\\
Considering $r^{\prime}$, we start by checking whether it codes a well-founded model of $ZF^{-}+V=L$. This can be done following the procedure described in the proof of the Lost Melody Theorem.\\ 
The next step is to determine whether that model is $\in$-minimal with the property that it contains a $\delta$-gap. For this,
let $r^{\prime-1}(a)$ denote the element of $L_{\gamma(\delta)}$ coded by $a$ for $a\in\omega$. For each pair $(a,b)\in\omega^{2}$, the set
$\{j\in\omega|r^{\prime-1}(j)\in2^{\omega}\cap L\wedge r^{\prime-1}(a)<_{L}r^{\prime-1}(j)<_{L}r^{\prime-1}(b)\}$ is computable from $r^{\prime}$. It remains to check for each of these pairs whether all elements in between are unrecognizable and whether 
the order type of the intervall is $\geq\delta$. The first can be done using the implementation of the truth-predicate given in the proof of the Lost Melody Theorem: We test for each natural number $i$ whether $i$ belongs to the intervall and represents an 
unrecognizable real number. If no recognizable real in the intervall is found, we return a positive answer, otherwise a negative answer. Checking whether $\delta$ embeds in the order type of the intervall can be done using the algorithm by Koepke from
\cite{WITRM}.\\
Finally, we check the $<_{L}$-minimality of $r^{\prime}$. This can be done in the same way as in the proof of the Lost Melody Theorem if $r^{\prime}\in L_{\gamma(\delta)+3}$. Since $\gamma(\delta)$ is minimal such that $L_{\gamma(\delta)}$ is a model of a
certain first-order theory, we will have $\Sigma_{\omega}^{L_{\gamma(\delta)+1}}\{\{L_{\gamma(\delta)}\}\}=L_{\gamma(\delta)+1}$, hence the proof goes through as there.\\
\end{proof}

The following answers both the question concerning the supremum of the $L$-stages containing new recognizable ordinals and the maximal size of gaps. Thanks to Philip Welch who suggested it in a short conversation after my talk at CiE $2012$, where I presented
Theorem \ref{gaps} and Theorem \ref{smallgaps} along with a sharper version showing that there are strong gaps of all lengths $\delta<\omega_{\omega}^{CK}$ (see \cite{Recog}).\\

\begin{defini}
A countable ordinal $\alpha$ is $\Sigma_{1}$-stable iff there exists a $\Sigma_{1}$-formula $\phi$ such that $\alpha$ is minimal with $L_{\alpha}\models\phi$.\\
\end{defini}

Since there are only countably many $\Sigma_{1}$-formulas, there are only countably many $\Sigma_{1}$-stable ordinals. Let $\sigma$ be the supremum of the $\Sigma_{1}$-stable ordinals. Then $\sigma$ is a countable ordinal.

\begin{thm}{\label{largegaps}}
(i) We have $RECOG\subset L_{\sigma}$.\\
(ii) $RECOG$ is cofinal in $L_{\sigma}$, i.e. for every $\alpha\in\sigma$, there is $\alpha<\beta<\sigma$ such that $RECOG\cap(L_{\beta+1}-L_{\beta})\neq\emptyset$.\\
(iii) There is no $\sigma$-gap.\\
(iv) For any $\delta<\sigma$, there is a strong gap of size $\delta$.
\end{thm}
\begin{proof}
(i) Let $x\in RECOG$. Hence there is some $ITRM$-program $P$ such that $P$ recognizes $P$. The statement $\phi(P):=\exists{x}P^{x}\downarrow=1$ is $\Sigma_{1}$. Hence, the smallest $\alpha$ such that $L_{\alpha}\models\phi(P)$ is $\Sigma_{1}$-stable.
Let $y\in L_{\alpha}$ be such that $L_{\alpha}\models P^{y}\downarrow=1$. By absoluteness of computations, it follows that $P^{y}\downarrow=1$. As $P$ recognizes $x$, we must have $y=x$. Hence $x\in L_{\alpha}\subset L_{\sigma}$.\\
(ii) Suppose $\alpha<\sigma$. Hence, there is some $\Sigma_{1}$-stable ordinal $\beta$ such that $\alpha<\beta<\sigma$. Since $\beta$ is $\Sigma_{1}$-stable, it follows from our lemma above that $L_{\beta+2}$ contains a real coding $L_{\beta}$.
Let $r$ be the $<_{L}$-smallest real coding $L_{\beta}$. Certainly $r\in L_{\beta+2}-L_{\beta}$ then. We claim that $r$ is recognizable:\\
To see this, let $\phi$ be a $\Sigma_{1}$-statement such that $\beta$ is minimal with $L_{\beta}\models\phi$. We can then re-use the strategy of the proof of Theorem \ref{LoMe} to check whether $r$ codes a minimal $L$-level in which $\phi$ holds.
The minimality of $r$ can then also be checked in the same way as in the proof of Theorem \ref{LoMe}.\\
(iii) Suppose $x\in RECOG$, where $P$ recognizes $x$. Let $\alpha$ be minimal such that $x\in L_{\alpha}$ and $L_{\alpha}\models P^{x}\downarrow=1$. Hence $\alpha$ is minimal such that $L_{\alpha}\models\exists{y}P^{y}\downarrow=1$. So
$\alpha$ is $\Sigma_{1}$-stable. Hence, $x$ is an element of a $\Sigma_{1}$-stable stage, and thus of $L_{\sigma}$.\\
(iv) Here, we use the idea of Theorem \ref{smallgaps}. Let $\delta<\sigma$. Let $\gamma$ be $\Sigma_{1}$-stable such that $\delta<\gamma$, $x:=cc(L_{\gamma})$. Furthermore, let $L_\alpha$ be minimal such that $L_{\alpha}\models ZF^{-}$ and contains a
strong $\gamma$-gap, $y:=cc(L_{\alpha})$. Finally, set $z:=x\oplus y$. We can recognize $x$ because $L_{\gamma}$ is $\Sigma_{1}$-stable. Now, given some real number $w$, we can check whether $w=z$ by first testing whether the even bits are those of 
$x$, and then whether $y$ codes a well-founded $L$-level modelling $ZF^{-}$ and finally, whether it has a gap of size $\gamma$ using the same strategy as for Theorem \ref{smallgaps}.
\end{proof}

This bounds $RECOG$ from above. The first Lost Melody appears right after the computable reals, as the following result shows:

\begin{thm}
 There exists a recognizable real in $L_{\omega_{\omega}^{CK}+2}-L_{\omega_{\omega}^{CK}}$.
\end{thm}
\begin{proof}
 Let $r=cc(L_{\omega_{\omega}^{CK}})$. As $r\notin L_{\omega_{\omega}^{CK}}$, $r$ is not computable. We claim that $r$ is recognizable. It is easy to check that $r$ codes an $L$-level which models 
'For each $i\in\omega$, $L_{\omega_{i}^{CK}}$ exists' and 'For all $x$, there is $i\in\omega$ such that $x\in L_{\omega_{i}^{CK}}$'. We can now check the $<_{L}$-minimality of $r$ as in the previous proofs, observing that
$\Sigma_{\omega}^{L_{\omega_{\omega}^{CK}+1}}\{L_{\omega_{\omega}^{CK}}\}=L_{\omega_{\omega}^{CK}+1}$.
\end{proof}

\begin{thm}
For each $\gamma<\sigma$, there are cofinally in $\sigma$ many $\alpha$ such that $L_{\alpha}\models ZF^{-}$, $[\beta,\alpha[$ is a weakly substantial strong gap for some $\beta<\alpha$ and $[\alpha,\alpha+\gamma[$ is a weakly substantial strong gap.
\end{thm}
\begin{proof}
$L_{\omega_{1}}$ models the following statements: $ZF^{-}+V=L$, 'there are cofinally many reals', 'there are cofinally many $\delta$ such that $L_{\delta}\models ZF^{-}$, the reals are cofinal in $L_{\delta}$, $RECOG\subset L_{\delta}$, $RECOG^{L_{\delta}}$ 
is bounded in $L_{\delta}$ and there are at least $\gamma$ many $L$-levels containing $L_{\delta}$'. (As we may take $\delta>\sigma$, this is easy to see.)\\
Taking the elementary hull $H$ of $\emptyset$ in $L_{\omega_{1}}$ and collapsing it via\\ $\pi:H\rightarrow L_{\eta}$, we see that $cc(L_{\eta})$ is a countable model of the same statements. Let $\eta^{\prime}$ be minimal with these properties,
then $cc(L_{\eta^{\prime}})$ is recognizable for the usual reasons and hence $\eta^{\prime}<\sigma$. Hence, we get gaps of the desired kind in $L_{\sigma}$ by absoluteness of non-recognizability.\\
To see that this happens cofinally often in $L_{\sigma}$, take some $\Sigma_{1}$-stable ordinal $\mu$ and consider the elementary hull of $\{\mu\}$ in $L_{\omega_{1}}$ to get a gap above $\mu$.
\end{proof}

\begin{thm}
 Let $L_{\alpha}\models ZF^{-}$. Then there is $\delta<\alpha$ such that $RECOG^{L_{\alpha}}\subseteq L_{\delta}$.
\end{thm}
\begin{proof}
We define the following function $f:\omega\rightarrow L_{\alpha}$: $f(i)=x$, if $P_{i}$ recognizes $x$ inside $L_{\alpha}$, i.e. if  $P^{i}(x)\downarrow=1$ and $P^{i}(y)\downarrow=0$ for all $x\neq y\in L_{\alpha}$, if such an $x$ exists.
Otherwise, let $f(i)=0(\in RECOG^{L_{\alpha}})$. It is easy to see that $f$ is definable by a $LAST$-formula.\\
Then, by replacement, $RECOG^{L_{\alpha}}=f[\omega]\in L_{\alpha}$. Hence, there is $\delta<\alpha$ such that $RECOG^{L_{\alpha}}\in L_{\delta+1}$, and thus $RECOG^{L_{\alpha}}\subseteq L_{\delta}$.
\end{proof}

Hence, the recognizables are bounded in $L$-levels modelling $ZF^{-}$. A careful analysis of the axioms actually needed reveals that something much weaker than $ZF^{-}$ is required here; in particular, replacement
is only required for $\Sigma_{3}$-formulas.\\

This gives us some information on the question where the first non-recognizable appears:

\begin{corollary}
Let $\alpha$ be minimal such that $L_{\alpha}\models ZF^{-}$. Then there is a non-recognizable real $x$ such that $x\in L_{\alpha}$.
\end{corollary}
\begin{proof}
By the last theorem, $RECOG^{L_{\alpha}}\in L_{\alpha}$. Furthermore, $f:\omega\leftrightarrow RECOG^{L_{\alpha}}$ is definable over $L_{\alpha}$. Hence $f$ is a definable subset of an element of $L_{\alpha}$ and therefore itself an
element of $L_{\alpha}$, since $L_{\alpha}\models ZF^{-}$. (Otherwise, for some $n\in\omega$, $\rho_{n}^{L_{\alpha}}<\alpha$, and the $\Sigma_{n}$-Skolem function $h_{n}$, which is definable over $L_{\alpha}$, maps $\rho_{n}$ surjectively
onto $L_{\alpha}$ and so, since $L_{\alpha}$ satisfies the replacement axiom, we get $L_{\alpha}\in L_{\alpha}$, a contradiction.) 
Thus we can define the usual diagonal function $d:\omega\rightarrow\{0,1\}$ by $d(i):=1-(f(i))(i)$ in $L_{\alpha}$. Then $d\in L_{\alpha}$, but $d$ is different from
all recognizable reals in $L_{\alpha}$.
\end{proof}

\textbf{Remark}: Following the arguments of \ref{largegaps}, we can further conclude that the first unrecognizable real is an element of $L_{\tilde{\sigma}}$, where $\tilde\sigma$ is the supremum of the $\Sigma_{1}$-stable ordinals in $L_{\alpha}$. 
With a bit more sophistication, this bound can be improved. We plan to investigate this in further work (see question $2$ below) and thank Philip Welch for some helpful suggestions in this direction.\\

\textbf{Question}: (1) When exactly does the first nonrecognizable appear? I.e. what is the minimal $\alpha$ such that $L_{\alpha}\cap\mathbb{R}\subsetneq RECOG$?\\
(2) More generally, let $\alpha>\omega_{\omega}^{CK}$ be such that $(L_{\alpha+1}-L_{\alpha})\cap\mathfrak{P}(\omega)\neq\emptyset$. Does it follow that $(L_{\alpha+1}-L_{\alpha})\cap(\mathfrak{P}(\omega)-RECOG)\neq\emptyset$? 

\subsection{Recognizable Sets of Reals}


\begin{defini}  
A set $X\subseteq\mathfrak{P}^{L}(\omega)$ is recognizable iff there is an $ITRM$-program $P$ such that, for all $r\in\mathfrak{P}(\omega)$, $P^{r}\downarrow=1$ iff $r\in X$, and otherwise $P^{r}\downarrow=0$. (Hence $\{r\}$ is recognizable iff
$r$ is.)
\end{defini}

The recognizable sets are closed under union, intersection and complementation and hence form a Boolean algebra. (To see closure under complementation, if $A$ is recognized by $P$, then $P^{\prime}$, which simply changes the output $x$ of $P$ to $1-x$ 
recognizes $\mathfrak{P}(\omega)-A$.) In \cite{KoeMi}, it is shown that every recognizable set
is $\Delta_{2}^{1}$ and that every $\Pi_{1}^{1}$-set is recognizable.\\
The crucial property for us is that every non-empty recognizable set of a certain kind contains a recognizable element.

\begin{defini}
Let $P$ be an $ITRM$-program. We call $x\subseteq\omega$ 'safe for $P$' iff $P^{x}(i)\downarrow$ for all $i\in\omega$.
\end{defini}

\begin{lemma}{\label{safetytest}}
Let $n\in\omega$ and let $(P_{i}|i\in\omega)$ be a canonical enumeration of $ITRM$-programs using at most $n$ registers. Then there is a program $Q$ such that, for all $x\subseteq\omega$, $Q^{x}(i)\downarrow=1$ iff $x$ is safe for $P_{i}$, and 
otherwise $Q^{x}(i)\downarrow=0$.
\end{lemma}
\begin{proof}
By \cite{KoeMi}, there is an $ITRM$-program solving the halting problem for $ITRM$-programs using at most $n$ registers, i.e. an $ITRM$-program $H$ such that $H^{x}(p(i,j))\downarrow=1$ iff $P_{i}^{x}(j)\downarrow$, and $H^{x}(p(i,j))\downarrow=0$, otherwise.
This can easily be modified to obtain the desired program $Q$, using the ability of $ITRM$s to search through the whole of $\omega$.
\end{proof}

\begin{lemma}{\label{containmenttest}}
There is an $ITRM$-program $P$ such that, for $x,y\subseteq\omega$, $P^{x\oplus y}\downarrow=1$ iff $y$ codes some $L_{\alpha}$ such that $x\in L_{\alpha}$ and $P^{x\oplus y}\downarrow=0$, otherwise.
\end{lemma}
\begin{proof}
By the proof of Theorem \ref{LoMe}, it is clear that we can check whether $y$ codes some $L_{\alpha}$. If that is the case, it remains to test whether $x\in L_{\alpha}$.\\
First, we claim that, for any $i\in\omega$, it is possible to compute a code for $i$ in $y$ uniformly in $y$ and $i$. This can be done using a straightforward recursion: To find a code $k_{0}$ for $0$, we
search through $\omega$ for some $k$ such that $p(j,k)\notin y$ for all $j\in\omega$. To find a code $k_{i+1}$ for $i+1$ given the codes $k_{0},...,k_{i}$ for $0,...,i$, we search for some $k\in\omega$ such that $p(j,k)\in y$ iff
$j\in\{k_{0},...,k_{i}\}$. It is easy to see that these searches can be carried out uniformly by an $ITRM$-program and yield the desired code.\\
Next, we claim that we can check, for some $i\in\omega$ whether $i$ codes $x$ in $y$ in the oracle $x\oplus y$. For this, let $j\in\omega$ be given. Use the first claim to find $k_{j}\in\omega$ coding $j$ in $y$. Now it is easy to 
check whether $j\in x\leftrightarrow p(k_{j},i)\in y$. Running through all $j\in\omega$ in this way, we obtain the desired checking procedure.\\
Finally, all that remains is to carry out this step for all $i\in\omega$ and return $1$ iff an appropriate $i$ is found and $0$ otherwise.
\end{proof}

This lemma will be applied in the following way: Let $P$ be an $ITRM$-program and let $x\subseteq\omega$ be safe for $P$. Then, there is a program $Q$ that decides whether or not $P^{x}$ computes a code of some $L$-level of which $x$ is an element.
Since it is clear that we can use the output of $P$ in the same way that we use an oracle, this follows immediately from the lemma.\\

\begin{defini}
Let $n\in\omega$. A real $x$ is called special iff $x\in L_{\omega_{1}^{CK,x}}$. The set of special reals is denoted by $SPECIAL$.
\end{defini}

\begin{lemma}{\label{specialunbounded}}
There are unboundedly many special real numbers, i.e. for every $\alpha<\omega_1$, there is a $1$-special real $x$ such that $x\notin L_{\alpha}$. 
\end{lemma}
\begin{proof}
Let $\alpha<\omega_1$ be arbitrary, and let $\gamma>\alpha$ be such that $(L_{\gamma+1}-L_{\gamma})\cap\mathfrak{P}(\omega)\neq\emptyset$, so that $\rho_{\omega}^{L_{\gamma}}=\omega$. Certainly, a code of a
well-ordering of order type $\gamma$ is recursive in $x$, so $\omega_{1}^{CK,x}>\gamma$. As $\omega_{1}^{CK,x}$ is a limit ordinal, it follows that $x\in L_{\omega_{1}^{CK,x}}$.
\end{proof}

\begin{lemma}{\label{containinglevels}}
There is $C\in\omega$ such that, for all special $x\in\mathfrak{P}^{L}(\omega)$, there is an $ITRM$-program $P$ using at most $C$ many registers such that $P^{x}$ computes a code for some 
$L_{\alpha}$ containing $x$.
\end{lemma}
\begin{proof}
As $x\in L_{\omega_{1}^{CK,x}}$ by assumption, $L_{\omega_{2}^{CK,x}}$ certainly contains such a code. To see this, observe that since $x\in L_{\omega_{1}^{CK,x}}$, we have $L_{\omega_{1}^{CK,x}}[x]=L_{\omega_{1}^{CK,x}}$, so 
$L_{\omega_{1}^{CK,x}}$ is the $\subseteq$-minimal admissible set containing $x$. In particular, it is the minimal $L$-level which contains $x$ and is a model of $KP$. This implies that 
$\Sigma_{\omega}^{\{L_{\omega_{1}^{CK,x}+1}\}}\{L_{\omega_{1}^{CK,x}}\}=L_{\omega_{1}^{CK,x}+1}$, so $L_{\omega_{1}^{CK,x}+2}$ contains a bijection between $\omega$ and $L_{\omega_{1}^{CK,x}}$ and hence already $L_{\omega_{1}^{CK,x}+3}$ will
contain a code for $L_{\omega_{1}^{CK,x}}$.\\
Hence, by Theorem \ref{compstrength1}, there is an $ITRM$-program computing such a code from $x$ using at most $g(2)$ many registers.
\end{proof}

This construction can be uniformized:\\

\begin{lemma}{\label{uniformcontaininglevels}}
There is a program $Q$ such that, for any special real $x$, $Q^{x}$ computes the code of some $L$-level containing $x$.
\end{lemma}
\begin{proof}
By Lemma \ref{containinglevels}, there is an $ITRM$-program computing such a code from $x$ using at most $C$ many registers. We now perform the following procedure for each $ITRM$-program $P$ using at most
$C$ registers: First, we check whether $x$ is safe for $P$. If not, we continue with the next program. Otherwise, $P^{x}$ will compute some real $y$. Using Lemma \ref{containmenttest}, we
can test whether or not $y$ codes an $L$-level containing $x$. If not, we continue with the next program. If it does, then we have found an $ITRM$-program computing a code of the desired kind from $x$ uniformly in $x$.
The desired code can than be computed by carrying out this program.
\end{proof}

\begin{lemma}{\label{recogspecial}}
$SPECIAL$ is recognizable.
\end{lemma}
\begin{proof}
Let a real $x$ be given. 

Using the strategy from the proof of Lemma \ref{uniformcontaininglevels}, we search for an $ITRM$-program using at most $C$ registers computing a code for an $L$-level containing $x$ uniformly from $x$.
If no such program is found, then $x$ is not special and we return $0$. 
Otherwise, let $\alpha$ be minimal such that $x\in L_{\alpha+1}$ and let $c$ be a code for $L_{\alpha+1}$ computable using at most $C$ registers. Clearly $x\in\omega_{1}^{CK,x}$ iff $\omega_{1}^{CK,x}\notin L_{\alpha+1}$. 
The next step is hence to search for (a code of) some ordinal in $L_{\alpha+1}$ in which every $x$-recursive well-ordering embeds. This can be done using the fact that $ITRM$s can decide whether the $e$th Turing program computes a well-ordering
in the oracle $x$ (uniformly in $x$ - see \cite{KoeMi}) and the procedure for checking whether one well-ordering embeds in another described in \cite{WITRM}.
If such an element is found, we output $0$, otherwise, we output $1$.
\end{proof}

\begin{lemma}{\label{recoghasrecog}}
Let $X\neq\emptyset$ be a recognizable set of special reals. Then \begin{center} $X\cap RECOG\neq\emptyset$. \end{center}
\end{lemma}
\begin{proof}
Let $\emptyset\neq X\subseteq SPECIAL$ be recognizable, and assume that $P$ recognizes $X$, where $P$ uses $n$ registers.  
As $X\subseteq SPECIAL$, we have $X\subseteq L$. Let $x$ be the $<_{L}$-minimal element of $X$. We claim that $x$ is recognizable.\\
To see this, note that there is an $ITRM$-program $Q$ that computes some $r$ coding an $L_{\alpha}$ with $x\in L_{\alpha}$ in the oracle $x$ by Lemma \ref{containinglevels}. As $x\in L_{\alpha}$, 
there is some $i\in\omega$ such that $i$ represents $x$ in $r$. To determine whether a certain given real $y$ is equal to $x$, we first check whether $y$ is special and safe for $Q$. By Lemma \ref{recogspecial} and Lemma \ref{safetytest}, 
this can be done uniformly in the oracle.
Then, using Lemma \ref{containmenttest}, we test whether $Q^{y}$ computes a code for an $L$-level in which $y$ is coded by $i$. 
If it doesn't, we output $0$. Otherwise, call $r^{\prime}$ the real calculated by $Q^{y}$. Now, we see whether $P^{y}\downarrow=1$ and output $0$ if this is not the case. 
Finally, we successively calculate $P^{z}$ for each special real coded in $r^{\prime}$ which is $<_{L}y$. (Note that these computations always terminate, because $P$ recognizes $X$. 
If the output is $1$ for some such real, we output $0$, otherwise, we output $1$. (Note that $P$ necessarily stops on each input, as it recognizes $X$.)
\end{proof}

\begin{corollary}{\label{unrecog}}
$SPECIAL\cap RECOG$ is not recognizable.
\end{corollary}
\begin{proof}
Otherwise, $SPECIAL-(SPECIAL\cap RECOG)$ is recognizable since $SPECIAL$ is recognizable and the recognizable sets are closed under intersection and complementation. By Lemma \ref{specialunbounded}, there are special reals
which are not elements of $L_{\sigma}$ and hence not recognizable. Hence $SPECIAL-(SPECIAL\cap RECOG)\neq\emptyset$. By Lemma \ref{recoghasrecog}, this set must hence contain a recognizable element, a contradiction.
\end{proof}

\subsection{Refinements}

A surprising feature of $ITRM$'s is that their computational strength increases with the number of registers; in fact, for every $n\in\omega$, there is $m\in\omega$ such that an $ITRM$ with $m$ registers can solve the halting problem for
$ITRM$-programs using $n$ registers uniformly in the oracle (see above). In particular, there is no universal $ITRM$. Here, we show that the same is true for the recognizability strength, i.e. that, for every $n\in\omega$, 
there is some real $x\in RECOG$ such that $x$ cannot be recognized by an $ITRM$-program using $n$ registers.

\begin{defini}
$RECOG_{n}$ denotes the set of reals recognizable by an $ITRM$-program $P$ using at most $n$ registers. Elements of $RECOG_{n}$ are called $n$-recognizable.
\end{defini}

\begin{thm}
There is $m\in\omega$ such that, for $n\geq m$, $RECOG_{n}$ is cofinal in $L_{\sigma}$, i.e., for every $n\geq m$ and $\alpha<\sigma$, we have\\ $(L_{\sigma}-L_{\alpha})\cap RECOG_{n}\neq\emptyset$.
\end{thm}
\begin{proof}
There exists a single program sufficient to check whether a certain real is the canonical code of the minimal $L$-stage modelling a certain $LAST$-statement by the proof of Theorem \ref{LoMe} and Lemma \ref{focheck}. 
This is enough to ensure cofinality in $L_{\sigma}$.
\end{proof}

\begin{thm}
Let $n\in\omega$. Then \begin{center} $SPECIAL\cap RECOG_{n}\subsetneq SPECIAL\cap RECOG$. \end{center} Consequently, $RECOG_{n}\subsetneq RECOG$.
\end{thm}
\begin{proof}
Assume for a contradiction that $n\in\omega$ is such that\\ $RECOG_{n}=RECOG$, i.e. that every recognizable real is recognizable by a program using at most $n$ registers. 
We claim that it follows that we can recognize $SPECIAL\cap RECOG$, which contradicts Lemma \ref{unrecog}.\\
Let $C$ be as in Lemma \ref{containinglevels}, set $m=max(C,n)$, and pick a canonical enumeration $(P_{i}|i\in\omega)$ of the $ITRM$-programs using at most $m$ registers. 
Now, let $x$ be an arbitrary real. We want to test whether $x\in SPECIAL\cap RECOG$. By Lemma \ref{recogspecial}, we can test whether $x$ is special. Let us hence from now on assume that it is. It remains to check whether $x$ is recognizable.
By Lemma \ref{safetytest}, let $R$ be an $ITRM$-program such that
$R^{x}(i)\downarrow=1$ iff $P_{i}^{x}(j)$ halts for all $j\in\omega$ and $R^{x}(i)\downarrow=0$, otherwise.\\
By Lemma \ref{uniformcontaininglevels}, we can (uniformly in $x$) compute a code $r$ for some $L_{\alpha}$ containing $x$.\\
Now, we search through $\omega$ for a program $P_{j}$ using at most $n$ registers with the following properties: (the number of registers used in $P_{i}$ can easily be computed from $i$)\\
(1) $P_{j}^{x}\downarrow=1$. This can be tested by using a halting problem solver for programs using at most $n$ registers (see Theorem \ref{haltingproblem}) and, in case $P_{j}^{x}\downarrow$, carrying out $P_{j}^{x}$.\\
(2) $P_{j}^{y}\downarrow=0$ for every special $y<_{L}x$. This can be done by searching through $\omega$ as the $<_{L}$-relation can be decided using $r$. Then, we apply step (1) to each such $y$.\\
If $x$ is recognizable and $P_{j}$ is a program using $n$ registers that recognizes $x$, then this search will be successful when the check is performed for $j$ at the latest. Hence, if the search is unsuccessful, then $x\notin RECOG_{n}=RECOG$.\\
On the other hand, if such a $j$ is found, then $x$ is $<_{L}$-minimal such that $P_{j}^{x}\downarrow=1$, while $P^{y}_{j}\downarrow=0$ for all $y<_{L}x$. Thus we can determine for some $z$ whether $z=x$ as follows: First check
whether $z$ is special. If it is, compute a code $r^{\prime}$ for some $L_{\alpha}\ni z$ as described above, (which is possible as $z$ is special), check whether $P_{j}^{z}\downarrow=1$ and then check whether $P_{j}^{y}\downarrow=0$ for every $y<_{L}z$.\\
Hence $x\in SPECIAL\cap RECOG$ iff the search is successful and thus, $SPECIAL\cap RECOG$ is recognizable, contradicting Lemma \ref{unrecog}.
\end{proof}

\textbf{Question}: Given $n\in\omega$, what is the smallest $m$ such that\\ $RECOG_{m}-RECOG_{n}\neq\emptyset$? A careful analysis of our proofs above shows that this $m$ will be dominated by some linear function of $n$.

\section{Relativization}

In the realm of recognizability, there is a natural analogue of the notion of relative computability, namely recognizability relative to an oracle. This section introduces this notion and gives some basic properties of the corresponding
reducibility relation.

\begin{defini}
Let $x\in\mathfrak{P}^{L}(\omega)$. Then $\alpha\in On$ is called $\Sigma_{1}^{(x)}$-stable iff there is a $\Sigma_{1}$-formula $\phi(v)$ with a parameter $v$ such that $\alpha$ is minimal with $x\in L_{\alpha}$ and $L_{\alpha}\models\phi(x)$.
We denote by $\sigma^{x}$ the supremum of the $\Sigma_{1}^{x}$-stable ordinals.
\end{defini}

\begin{defini}
 Let $x,y\in\mathfrak{P}(\omega)\cap L$. We say that $x$ is recognizable in $y$ iff there is an $ITRM$-program $P$ such that, for all $z\in\mathfrak{P}(\omega)\cap{L}$, $P^{z\oplus y}\downarrow=1$ iff $z=x$ and otherwise $P^{z\oplus y}\downarrow=0$.
We write $x\leq_{RECOG}y$ if $x$ is recognizable in $y$.\\
If $x\leq_{RECOG}y$, but not $y\leq_{RECOG}x$, we write $x<_{RECOG}y$.
\end{defini}

Most of our results above relativize in a straightforward way. We summarize them in the following theorem:

\begin{thm}{\label{basrel}}
Let $x\in\mathfrak{P}^{L}(\omega)$.\\
(i) There is a real $r$ computable in $x$, but not recognizable in $x$.\\
(ii) For every $\delta<\sigma^{x}$, there is a $\delta$-gap in the $x$-recognizables.\\
(iii) For $x\in\mathfrak{P}^{L}(\omega)$, $\sigma^{x}$ is the minimal $\alpha$ such that $RECOG^{x}\subseteq L_{\alpha}$.\\
(iv) If $x$ is computable in the oracle $y$, then $x\leq_{RECOG}y$.
\end{thm}
\begin{proof}
All of these are proved by obvious generalizations of the proofs for the corresponding non-relativized statements.
\end{proof}

\begin{prop}
 For any $x\in\mathfrak{P}^{L}(\omega)$, there is $y\in\mathfrak{P}^{L}(\omega)$ such that $x<_{RECOG}y$.
\end{prop}
\begin{proof}
By (iii) of the last theorem, we know that, for any $x$, $RECOG^{x}\subset L_{\sigma^{x}}$. As $\sigma^{x}<\omega_{1}$, there must be $y\in\mathfrak{P}^{L}(\omega)-L_{\sigma^{x}}$, and all these $y$ will not be recognizable in $x$.
So for such a $y$ and we have $y\not\leq_{RECOG}x$. Now pick some $\alpha>\sigma^{x}$ such that $L_{\alpha}\models ZF^{-}$ and set $z:=cc(L_{\alpha})$. Then $z\notin L_{\sigma^{x}}$, hence $z\not\leq_{RECOG}x$. But, on the other hand, $x\leq_{RECOG}z$ is obvious.
(In fact, of course every real in $L_{\alpha}$ is recognizable in $cc(L_{\alpha})$.) Hence $x\leq_{RECOG}z$, but $z\not\leq_{RECOG}x$, so $x<_{RECOG}z$.
\end{proof}

Relative computability and relative recognizability are linked in an obvious way:

\begin{prop}
 Let $x\leq_{ITRM}y$ denote the statement that $x$ is $ITRM$-computable in the oracle $y$. Let $x,y,z\in\mathfrak{P}(\omega)$ such that $x\leq_{RECOG}y$ and $y\leq_{ITRM}z$. Then $x\leq_{RECOG}z$.
\end{prop}
\begin{proof}
All information about $y$ relevant for identifying $x$ can be computed from $z$ by the assumption $y\leq_{ITRM}z$.
\end{proof}

In the other direction, this statement is false:

\begin{prop}{\label{nottrans}}
 There are $x,y,z\in\mathfrak{P}^{L}(\omega)$ such that $x\leq_{ITRM}y\leq_{RECOG}z$, but $x\not\leq_{RECOG}z$. 
\end{prop}
\begin{proof}
 Let $x$ be the $<_{L}$-smallest nonrecognizable real and let $y=cc(L_{\alpha})$, where $\alpha$ is minimal with $L_{\alpha}\models ZF^{-}$. Above we proved that $x\in L_{\alpha}$. Now it is easy to see that $x$ is computable from $y$: In $y$, $x$ must be coded by
some natural number $n$. For every natural number $i$, we can determine a natural number $j$ coding it in $y$. Hence, we can check for every $k\in\omega$ whether or not $k\in x$, given $cc(L_{\alpha})=y$. (In fact, the argument shows that every real in $L_{\alpha}$ 
is computable in the oracle $y$.)\\
Now $y\in RECOG$, hence $y\leq_{RECOG}0$. But obviously, $x\not\leq_{RECOG}0$, as $x$ is unrecognizable.
\end{proof}

It might be tempting to define $x\sim_{RECOG}y$ iff $x\leq_{RECOG}y$ and $y\leq_{RECOG}x$ and consider in this way 'recognizability degrees'. Unfortunately, by Proposition \ref{nottrans} and part d) of Theorem \ref{basrel}, $\leq_{RECOG}$ is not 
transitive, and hence $\sim_{RECOG}$ defined in this way will not be an equivalence relation. \\
Nevertheless, we can formulate some concepts and pose some questions typical for a degree theory.

\begin{defini}
Let $x,y\in\mathfrak{P}^{L}(\omega)$. $x$ and $y$ are $RECOG$-comparable iff $x\leq_{RECOG}y$ or $y\leq_{RECOG}x$. Otherwise, they are $RECOG$-incomparable.
\end{defini}

Interestingly, this reducibility relation becomes tamer if we iterate it:

\begin{defini}
 Let $x\leq_{RECOG}^{2}y$ iff there is $a\in\mathfrak{P}^{L}(\omega)$ such that $x\leq_{RECOG}a\leq_{RECOG}y$. If $x\leq_{RECOG}^{2}y$, we say that $x$ is $2$-recognizable in $y$.
\end{defini}

\begin{thm}
$x\leq_{RECOG}^{2}y$ if $x\in L_{\sigma^{y}}$.
\end{thm}
\begin{proof}
Let $x\in L_{\sigma^{y}}$ be a real and pick $\alpha<\sigma^{y}$ such that $\alpha$ is $\Sigma_{1}^{y}$-stable and $x\in L_{\alpha}$. Arguing as in the proof of the Lost Melody Theorem, $cc(L_{\alpha})$ is recognizable in $y$.\\
Furthermore, we see that $x\leq_{ITRM} cc(L_{\alpha})$: Let $j$ code $x$ in $cc(L_{\alpha})$ and let $i\in\omega$ be arbitrary. To test whether $i\in x$, we identify a code $k$ for $i$ in $cc(L_{\alpha})$ in the way described in the proof
of Lemma \ref{containmenttest} and then check whether $p(k,j)\in cc(L_{\alpha})$.\\
Hence, we have $x\leq_{RECOG} cc(L_{\alpha})$ by (iv) of Theorem \ref{basrel}. So $x\leq_{RECOG}cc(L_{\alpha})\leq_{RECOG}y$, thus $x\leq_{RECOG}^{2}y$, as desired.
\end{proof}

\noindent
\textbf{Remark}: In particular, every real in $L_{\sigma}$ is $2$-recognizable in the empty oracle.\\

\noindent
\textbf{Questions}: (1) Is there a pair of $RECOG$-incomparable constructible reals?\\
(2) Is there, for any constructible reals $x<_{RECOG}y$, a constructible real $z$ such that $x<_{RECOG}z<_{RECOG}y$?

\newpage

\section{Acknowledgements}
We are indebted to Philip Welch, who, in two short conversations on CiE $2012$, suggested Theorem \ref{largegaps} as an answer to an open question and the main ideas for the proofs of Theorem \ref{shoenfield} and Theorem \ref{mansfield}. He also gave some helpful
corrections and suggestions on an earlier version of this paper. We also thank Joel Hamkins for suggesting the proof of Lemma \ref{specialunbounded}.


\begin{thebibliography}{CDK}

\bibitem[Cut]{Cutl} N.J. Cutland. Computability: An Introduction to Recursive Function Theory. Cambridge University Press (1980)

\bibitem[Mi]{Mi} A.W. Miller. Descriptive Set Theory and Forcing: How to prove theorems about the Borel sets the hard way. Available online: http://www.math.wisc.edu/~miller/res/dstfor.pdf 

\bibitem[ITTM]{ITTM} J.D. Hamkins and A. Lewis. Infinite Time Turing Machines. J. Symbolic Logic, 65(2), 567-604 (2000)

\bibitem[WITRM]{WITRM} P. Koepke. Infinite Time Register Machines. In A. Beckmann et al. (eds.) Logical approaches to computational barriers, LNCS, vol 3988, pp. 257-266. Springer, Heidelberg (2006) 

\bibitem[Koe]{Koe} P. Koepke. Ordinal computability. In Mathematical Theory and Computational Practice. K. Ambos-Spies et al, eds., Lecture Notes in Computer Science 5635 (2009), 280-289.

\bibitem[KoeMi]{KoeMi} P. Koepke and R. Miller. An Enhanced Theory of Infinite Time Register Machines. In A. Beckmann et al. (eds) Logic and Theory of Algorithms, LNCS, vol. 5028, pp 306-315 (2008)

\bibitem[KoeWe]{KoeWe} P. Koepke and P. Welch. A Generalized Dynamical System, Infinite Time Register Machines, and $\Pi_{1}^{1}-CA_{0}$. In CiE 2011. B. L\"owe et al. (eds.), LNCS
6735, 152-159 (2011)

\bibitem[Je]{Jech} T. Jech. Set Theory. 3rd Millenium edition, revisited and expanded. Springer (2002)

\bibitem[ITRM]{ITRM} M. Carl, T. Fischbach, P. Koepke, R. Miller, M. Nasfi, G. Weckbecker. The basic theory of infinite time register machines. Archive for Mathematical Logic 49 (2010) 2, 249-273

\bibitem[Jen]{Jen} R.B. Jensen. The fine structure of the constructible hierarchy. Ann. Math. Logic (1972)

\bibitem[Ca1]{Carl} M. Carl. Alternative Finestructural and Computational Approaches to Constructibility. PhD Thesis. Bonn, 2010.

\bibitem[Ca2]{Recog} On the distribution of recognizable ordinals. Notes for a talk at CiE 2012. 

\bibitem[HamLew]{HamLew} J. Hamkins, A. Lewis. Infinite Time Turing Machines. J. Symbolic Logic, vol. 65, iss. 2, pp. 567-604, 2000. 

\bibitem[Si]{Si} S.G. Simpson. Subsystems of Second Order Arithmetic. Cambridge University Press. $2009$.
\end{thebibliography}
\end{document}